\documentclass[11pt]{amsart}

\usepackage[utf8]{inputenc}

\usepackage{amsmath,amsthm,verbatim,amssymb,amsfonts,amscd, graphicx,color, framed,enumerate,stackrel}
\usepackage[left=1in,right=1in,top=1in,bottom=1in]{geometry}
\usepackage[colorlinks]{hyperref}
\usepackage{tikz}
\usepackage[normalem]{ulem}

\usepackage{pgfplots}
\pgfplotsset{width=10cm,compat=1.9}

\def\R{\mathbb R}
\newcommand{\RR}{\mathbb{R}}

\def\Im{\text{Im}}

\def\dim{\text{dim}}

\theoremstyle{definition}
\newtheorem{theorem}{Theorem}[section]
\newtheorem{lemma}[theorem]{Lemma}

\newtheorem{proposition}[theorem]{Proposition}

\newtheorem{corollary}[theorem]{Corollary}

\newtheorem{definition}[theorem]{Definition}

\newtheorem{example}[theorem]{Example}

\newtheorem{remark}[theorem]{Remark}

\title[Connectedness of multistationarity regions of small reaction networks]{On the connectedness of multistationarity regions\\ of small reaction networks }
\author{Allison McClure and Anne Shiu (Texas A\&M University)}

\date{April 5, 2024}

\begin{document}

\maketitle

\begin{abstract}
A multistationarity region is the part of a reaction network's parameter space that gives rise to multiple steady states. Mathematically, this region consists of the positive parameters for which a parametrized family of polynomial equations admits two or more positive roots.  Much recent work has focused on analyzing multistationarity regions of biologically significant reaction networks and determining whether such regions are connected; indeed, a better understanding of the topology and geometry of such regions may help elucidate how robust multistationarity is to perturbations.  Here we focus on the multistationarity regions of small networks, those with few species and few reactions.  For two families of such networks -- those with one species and up to three reactions, and those with two species and up to two reactions -- we prove that the resulting multistationarity regions are connected.  We also give an example of a network with one species and six reactions for which the multistationarity region is disconnected. Our proofs rely on the formula for the discriminant of a trinomial, a classification of small multistationary networks, and a recent result of Feliu and Telek that partially generalizes Descartes' rule of signs.

\vskip 0.1in

\noindent
{\bf Keywords:}
steady state, 
multistationary, 
reaction network, mass-action kinetics, discriminant %, Descartes' rule of signs

\noindent
{\bf MSC Codes:}
37N25,  %Dynamical systems in biology
92E20, %Classical flows, reactions, etc. in chemistry
12D10, %Polynomials in real and complex fields: location of zeros (algebraic theorems)
37C25 %Fixed points and periodic points of dynamical systems; fixed-point index theory; local dynamics
\end{abstract}

\section{Introduction} \label{sec:intro}
This work focuses on the question, {\em Which reaction networks give rise to multistationarity regions that are connected?}  Mathematically, this question translates to the following: For certain systems of polynomials in variables $x_1,x_2,\dots, x_n$, and
involving positive real parameters 
$\kappa_1,\kappa_2,\dots, \kappa_r$:
%is a vector of (positive real) parameters, 
    \begin{align*}
     f_i(\kappa;x)~=~0 \quad \quad {\rm for~} i=1,2,\dots, n~,
    \end{align*}
%$\{ f_i(\kappa;x)=0 \mid i=1,2,\dots, n\}$, where 
when is the set of parameter vectors $\kappa$ for which the system admits more than one positive real root $x \in \mathbb{R}^n_{>0}$, connected?  
%Hence, this is fundamentally a quantifier elimination problem pertaining to semi-algebraic sets.

\subsection{Motivation}
%The motivation for this question is as follows.  
Reaction networks arise in many applications -- including systems biology, ecology, and epidemiology -- and one key question is whether the resulting dynamical systems are multistationary (that is, admit multiple steady states).  In applications, multistationarity (or, more precisely, multistability) is the foundation for cellular switch-like and decision-making behavior~\cite{laurent}.  Accordingly, much work has focused on the question of which networks are multistationary~\cite{mss-review}.

Given a multistationary network, an important follow-up problem is
to describe the {\em multistationarity region}: the set of all parameters (reaction rate constants and/or conservation-law values) for which the corresponding dynamical system is multistationary.  
Understanding this ``geography of parameter space''~\cite{parameter-geography} -- for instance, are these regions open (and hence full dimensional) or connected? -- 
has attracted much attention in recent years and may help us understand how robust multistationarity and multistability are to perturbations~\cite{telek-feliu-topological}.  

\subsection{Our contribution}
%As noted above, a reaction network's parameter space can be viewed in terms of the reaction rate constants, its conservation-law values, or both.  We are interested in two such options, and we carefully distinguish between them when defining multistationarity regions (see Definition~\ref{def:multistationarity-region}), even though both concepts are called ``multistationarity regions'' in the literature.  

Our main result pertains to small reaction networks, as follows.
%--------------------------------------------
% SUMMARY THEOREM
%--------------------------------------------
\begin{theorem} \label{thm:summary}
If $G$ is a network with (i)
exactly one species and up to three reactions or
(ii)~exactly two species and up to two reactions,
%one of the following:
%    \begin{enumerate}
%        \item exactly one species and up to two reactions, or
%        \item exactly two species and two reactions,
%    \end{enumerate}
then the multistationarity region of $G$ is connected.
    \end{theorem}
Theorem~\ref{thm:summary} encompasses Theorems~\ref{thm:1-species-up-to-3-rxns} and~\ref{thm:2-species} in the main text.  
We also give an example of a multistationarity region that is disconnected, which arises from a network with only one species and six reactions (Proposition~\ref{prop:disconnected}). 
To our knowledge, only one other network in the literature has been shown to exhibit a disconnected multistationarity region~\cite{telek-feliu-topological}.  
Additionally, for all networks considered in this work, we completely describe the multistationarity region (some of these regions are listed in Table~\ref{tab:networks} in Section~\ref{sec:multi-region}), a task that is generally infeasible for medium- to large-size networks (say, $10$ or more reactions).  
Our main results are summarized in Table~\ref{tab:summary}.

\begin{table}[ht]
	\begin{center}
		\begin{tabular}{l c c}
			\hline
			Number of species & Number of reactions & Is multistationarity region connected?\\
			\hline
			%-------------------
			1 & at most $3$ & Yes \\
			%-------------------
			1 & $6$ or more & Yes or No \\
			%-------------------
			2 & at most $2$ & Yes \\
			\hline
		\end{tabular}
	\end{center}
	\caption{Summary of main results (Proposition~\ref{prop:disconnected} and Theorems~\ref{thm:1-species-up-to-3-rxns} and~\ref{thm:2-species}). \label{tab:summary}}
\end{table}%

Our proofs rely on a classification of small multistationary networks due to Joshi and Shiu~\cite{joshi-shiu-small} and on two results pertaining to polynomials: a formula for the discriminant of a trinomial~\cite{swan} and 
a result of Feliu and Telek that  
partially generalizes Descartes' rule of signs~\cite{TelekDescartes}.

\subsection{Relation to literature}
Other researchers have also analyzed the multistationarity regions of small reaction networks.  
 For instance, Joshi computed the (unique) inequality defining the  multistationarity regions of fully open, one-species networks  with only one non-flow reaction, such as $\{0 \rightleftarrows A,~ 2A \to 3A\}$~\cite{joshi2013complete}.  
Subsequently, Helmer and Feliu used Gale duality to extend Joshi's analysis 
to allow for any number of species (again in networks with only one non-flow reaction)~\cite{feliu-helmer}; connectedness of the multistationarity region, however, is not immediate from their results.  Related results were proven by Tang, Lin, and Zhang~\cite[Theorem 6.1]{multi-1-d}.  Our work significantly extend Joshi's result (see Example~\ref{ex:joshi-again}), but our results do not touch upon the related works~\cite{feliu-helmer, multi-1-d}. 

As mentioned earlier, a complete description of the multistationarity region is generally infeasible (or at least unwieldy) for 
networks of biologically realistic sizes.  One approach, therefore, is to 
establish open subsets within multistationarity regions.  
This has been accomplished for certain biochemically significant networks, by using algebraic techniques that harness the structure and sparsity appearing in polynomial systems arising in many biochemical networks~\cite{bihan-dickenstein-giaroli, conradi2019multistationarity, giaroli-rischter-mpm-dickenstein, bistab-seques}. 
%, such as Bihan, Dickenstein, Giaroli, Richster, and P\'erez Mill\'{a}n, 
Another approach, 
due to Sadeghimanesh and England, is to approximate multistationarity regions by polynomial super-level sets~\cite{amir-superlevel}.
A related attempt to gain a detailed understanding of such regions was pursued
by Bradford {\em et al.}, who used computational (symbolic and numerical) techniques~\cite{bradford2020identifying}.

Finally, as noted before, much recent interest has focused on the question of whether 
multistationarity regions are connected.  
Feliu, Kaihnsa, de Wolff, and Y\"{u}r\"{u}k showed that this region\footnote{More precisely, this region is what we call the ``multistationarity-allowing region'' (Definition~\ref{def:multistationarity-region}). } is indeed connected for an important biological signaling network -- the dual-site phosphorylation cycle with sequential and distributive mechanisms~\cite[\S5]{FKdY} -- and then subsequently generalized this result significantly~\cite{FKdY2}.  Their result proves rigorously what was strongly suggested by a recent numerical (as opposed to symbolic) investigation of the same system by Nam {\em et al.}~\cite{parameter-geography}. (The difference in what was analyzed concerns multiple steady states versus multiple {\em exponentially stable} steady states, but – up to a set of measure zero – these regions are expected to coincide). 

Additional results pertaining to connectedness of multistationarity regions are due to Telek and Feliu~\cite{TelekDescartes, telek-feliu-topological}. Notably, they give an algorithm that can assert connectedness.  However, in some cases, the algorithm is inconclusive (as is the case for many of the networks analyzed in our work). %E.g. has boundary steady states

\begin{remark}\label{rem:path-conn}
All occurrences of ``connected'' in our work can be replaced by ``path-connected''.  Indeed, the two concepts are equivalent for open subsets of Euclidean space, and all multistationarity regions considered in our examples and results are open (in fact, defined by strict inequalities).
% locally path connected
\end{remark}

\subsection{Organization of article}
%This article is organized as follows. 
This article has two background sections: 
one on 
polynomials (Section~\ref{sec:background-polynomials}) 
and one on
reaction networks (Section~\ref{sec:background}). 
In Section~\ref{sec:multi-region}, we define multistationarity regions, and then our
main results appear in
Section~\ref{sec:results}.
We end with a discussion in Section~\ref{sec:discussion}.

\section{Background on polynomials} \label{sec:background-polynomials}
This section recalls results on polynomials in one variable (Section~\ref{sec:univariate}) and several variables (Section~\ref{sec:multivariate}).

\subsection{Univariate polynomials} \label{sec:univariate}
The following formula for the discriminant of a univariate trinomial (that is, a one-variable polynomial with only three monomials) is due to Swan~\cite{greenfield-drucker, swan}.
%(see also~\cite{greenfield-drucker}).

\begin{lemma}[Discriminant of trinomial]  \label{lem:trinomial-discriminant}
The trinomial \(g(x) = x^{n} + ax^k + b\), where \(0 < k < n\), has {discriminant} 
\begin{equation*}
    %D(g) = 
    (-1)^{n(n-1)/2}b^{k-1} 
        \left[ 
        n^{N}b^{N-K}-(-1)^{N}(n-k)^{N-K}k^{K}a^{N} \right] ^d~,
\end{equation*}
where \(d = \gcd(n,k)\), \(N = n/d\), and \(K=k/d\). 
\end{lemma}

\begin{example} \label{ex:discriminant-quadratic}
    For a quadratic $x^2+bx+c$, the formula for the discriminant given in Lemma~\ref{lem:trinomial-discriminant} yields the standard discriminant: $b^2-4c$.
\end{example}

Next, we use Lemma~\ref{lem:trinomial-discriminant} to characterize the number of positive roots of a trinomial in which the coefficients alternate in sign.  We use this result in later sections.

\begin{proposition} \label{prop:num-roots-trinomial}
Consider a trinomial \(g(x) = x^{n} - cx^k + b\) with \(0 < k < n\), where $b>0$ and $c>0$.  Let $\mathfrak N(g)$ denote the number (counted without multiplicity) of positive roots of $g$, and let:
    \[ D(g)~:=~ n^{N}b^{N-K}-(n-k)^{N-K}k^{K}c^{N}~,\] 
    where
 \(d = \gcd(n,k)\), \(N = n/d\), and \(K=k/d\). 
 Then:
\begin{align*}
    \mathfrak N(g)
    ~=~ 
    \begin{cases}
        2 \quad & 
            \textrm{if } D(g)<0 \\
        1 \quad &
            \textrm{if } D(g) = 0 \\ 
        0 \quad &
            \textrm{if } D(g) > 0~.    
    \end{cases}
\end{align*}

\end{proposition}

\begin{proof} 
Let 
$g$ be as in the statement of the proposition. 
By applying Lemma~\ref{lem:trinomial-discriminant} and using the fact that $b>0$, we see that the discriminant of $g$ is $0$ if and only if 
$D(g)= n^{N}b^{N-K}-(n-k)^{N-K}k^{K}c^{N}  $ is also $0$. 

Our proof proceeds by analyzing what happens to $g(x)$ as $b$ varies (we view $c$ as fixed). 
First consider 
the case of $b=0$.  In this case, 
%Observe that when $b=0$, 
$D(g)<0$ and
(by a calculus exercise) the graph of $g(x)=x^n-c x^k=x^k(x^{n-k} -c)$, restricted to $x \geq 0$, 
satisfies the following: 
(1) there are 
two roots, at $0$ and at $\sqrt[n-k]{c}$, 
and
(2) 
the graph 
is decreasing on an interval $[0,\alpha]$ (with $\alpha < \sqrt[n-k]{c}$) and is 
increasing on the subsequent interval $[\alpha,\infty)$.  In other words, 
the graph of $g$ (for $x \geq 0$) has the form:

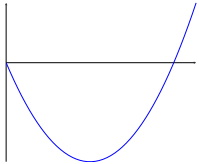
\begin{center}
\begin{tikzpicture}[scale=0.3]
\begin{axis}[
    axis lines = center,
    ticks=none
]
\addplot [
    domain=0:1.7, 
    samples=50, 
    color=blue,
]
{x^2 - 1.5*x};
\end{axis}
\end{tikzpicture}
\end{center}

\noindent
As $b$ increases, the above graph shifts up (by $b$).  At first, there are two positive roots, and then they become a double root, at which point the discriminant of $g$ and hence also $D(g)$ is $0$.  This happens exactly when $b$ takes the following value:
\[
b^* ~=~
    \sqrt[N-K]{ \frac{(n-k)^{N-K}k^{K}(c^*)^{N}}{n^{N}} }~.
\]
For larger $b$ (that is, $b> b^*$), 
we have $D(g)>0$ and
there are no positive roots.  Now our desired result follows directly.
\end{proof}

\subsection{Multivariate polynomials} \label{sec:multivariate}
Following~\cite{TelekDescartes}, a 
{\em polynomial function} denotes a function $g:\mathbb{R}_{>0}^{n} \rightarrow \mathbb{R}$ obtained by restricting a multivariate polynomial to the positive orthant:
    \[
    g(x)~=~ c_1 x_1^{\mu_{11}} x_2^{\mu_{12}} \dots x_n^{\mu_{1n}}
            ~+~ 
            c_2 x_1^{\mu_{21}} x_2^{\mu_{22}} \dots x_n^{\mu_{2n}}
            ~+~ 
            \dots 
            ~+~ 
            c_{\ell} x_1^{\mu_{\ell 1}} x_2^{\mu_{\ell 2}} \dots x_n^{\mu_{\ell n}}~,
    \]
where $c_i \in \mathbb R $ and $\mu_{ij} \in \mathbb{Z}_{\geq 0}$
for all $i \in \{1,2,\dots, \ell\}$ and $j \in \{1,2,\dots, n\}$.
The following result pertaining to polynomial functions is due to Feliu and Telek \cite[Theorem~3.4]{TelekDescartes}:
\begin{lemma} \label{lem:signomial}
Let \( g:\mathbb{R}_{>0}^{n} \rightarrow \mathbb{R}\) be a 
polynomial function.  If at most one coefficient of $g$ is negative, then \(g^{-1}(\mathbb{R}_{<0})\) is connected.
\end{lemma}

%-------------------
%    SECTION
%-------------------
\section{Background on reaction networks} \label{sec:background}
This section recalls reaction networks and mass-action systems (Section~\ref{sec:network}) 
and their capacity for multiple steady states (Section~\ref{sec:mss}).

\subsection{Mass-action systems} \label{sec:network}
As a preview to formal definitions, we begin with an example, which  serves as a running example in this work.  The following {\em reaction network} consists of two {\em reactions} involving the {\em species} $A$ and $B$:
\begin{align} \label{eq:2-species}
	\{ 2A + B \to 3A,~ A \to B\}~.
\end{align}
In examples, as above, we write species as $A,B,C,\dots$.  However, for precise definitions, it is more convenient to write $X_1,X_2,X_3,\dots$, as follows.

A {\em reaction network} $G$ consists of a finite set of {\em reactions}, as follows:
\begin{align*}
	G ~=~ 
	\left\{
	(y_{i1}X_1 + y_{i2} X_2 + \dots + y_{in} X_n) 
	\to 
	(y_{i1}' X_1 + y_{i2}' X_2 + \dots + y_{in}' X_n) 
	\mid 
	i=1,2, \dots, r
	\right\}~,
\end{align*}
where all $y_{ij}$ and $y_{ij}'$ are nonnegative integers, and $X_1, X_2, \dots, X_n$ are {\em species.}  For a {\em complex} 
$y_{i1}X_1 + y_{i2} X_2 + \dots + y_{in} X_n $, we use the shorthand 
$y_i=( y_{i1}, y_{i2}, \dots, y_{in})$.  For a reaction $y_i \to y'_i$, the complex $y_i$ is the {\em reactant} and $y_i'$ is the {\em product}.

Next, we describe how a reaction network $G$ defines, via mass-action kinetics, a system of ordinary differential equations (ODEs). 
Let $r$ denote the number of reactions of $G$, and let $y_i \to y_i'$ denote the $i$-th reaction.
Let $\kappa=(\kappa_1, \kappa_2, \dots, \kappa_r) \in \mathbb{R}^r_{>0}$ denote a vector of positive {\em rate constants} (one rate constant per reaction).  
The {\em mass-action system} arising from $G$ and $\kappa$, which we denote by $(G,\kappa)$, is the dynamical system defined by the following ODEs:
\begin{align} \label{eq:mass_action_ODE}
		\frac{dx}{dt} ~=~ \sum_{i=1}^r  \kappa_{i} x^{y_i} (y_i'-y_i) ~=:~ f_{\kappa}(x)~,
\end{align}
where  $x_i(t)$ denotes the concentration of the species $X_i$ at time $t$, and $x^{y_i} := \prod_{j=1}^n x_j^{y_{ij}}$.

Next, the right-hand sides of the ODEs~\eqref{eq:mass_action_ODE} always lie in the linear subspace of $\mathbb{R}^n$ spanned by all reaction vectors $y_i' - y_i$ (for $i=1,2,\dots, r$).   This subspace is the {\em stoichiometric subspace} of $G$, and we denote it by $S$.
A network is {\em full dimensional} if $S= \mathbb{R}^n$.
 
Another property of the ODEs~\eqref{eq:mass_action_ODE} is that forward-time solutions $\{x(t) \mid t \ge 0\}$ that begin in the nonnegative orthant $\RR_{\geq 0}^n$, remain in 
$\RR_{\geq 0}^n$.  
Hence,  
a solution $\{x(t) \mid t \ge 0\}$ of~\eqref{eq:mass_action_ODE}, with initial condition $x(0)  \in \R_{>0}^n$, stays in the following {\em stoichiometric compatibility class}: 
\begin{align} \label{eq:SCC}
%P_{x(0)} 
P ~=~ (x(0)+S)\cap \RR_{\geq 0}^n~.
\end{align}

\begin{example} \label{ex:ODEs}
The network~\eqref{eq:2-species} generates the following mass-action ODEs~\eqref{eq:mass_action_ODE}:  
	\begin{align} \label{eq:ODEs-running-example}
    \frac{dx_1}{dt} ~&=~ \kappa_{1}x_{1}^{2}x_{2} - \kappa_{2}x_{1}  \\
    \frac{dx_2}{dt} ~&=~ -\kappa_{1}x_{1}^{2}x_{2} + \kappa_{2}x_{1}~. \notag
    \end{align}
The stoichiometric subspace is one-dimensional, spanned by the vector $(1,-1)^\top$. 
The stoichiometric compatibility class~\eqref{eq:SCC} defined by the initial condition 
$x(0)= (1,1.5)$ is as follows:
\begin{align} \label{eq:SCC-for-2-species}
%P_{(1,1.5)} 
P ~=~ \{ (x_1,x_2) \in \mathbb{R}^2_{\geq 0} \mid x_1 + x_2 = 2.5 \}~.
\end{align}
\end{example}

The equation 
$x_1 + x_2 = 2.5$, 
in~\eqref{eq:SCC-for-2-species},
can be viewed as a conservation law.  Accordingly, we reframe stoichiometry-related concepts in terms of conservation laws, as follows.  A {\em conservation-law matrix} of $G$, denoted by $W$, refers to a $(d\times n)$-matrix whose rows are a basis of the orthogonal complement of $S$ (here, $d=n - \dim(S)$).  Now the stoichiometric compatibility class~\eqref{eq:SCC} can be rewritten: %as 
\begin{align} \label{eq:SCC-rewritten}
P_{c} ~=~ \{x\in {\mathbb R}_{\geq 0}^n \mid Wx=c\}~,
\end{align}
where $c:= W x(0) \in {\mathbb R}^d$ is called a
{\em total-constant vector} (or {\em total-concentration vector}).  

\begin{example}[Example~\ref{ex:ODEs}, continued] \label{ex:cons-law}
Returning to the network
$\{ 2A + B \overset{\kappa_1}{\to} 3A,~ A \overset{\kappa_2}{\to} B\}$, a conservation-law matrix is the $(1 \times 2)$-matrix
$W=[1~1]$.  For $x(0)= (1,1.5)$, we have $c= W x(0)=2.5$, so the 
stoichiometric compatibility class found earlier~\eqref{eq:SCC-for-2-species} matches the description of $P_c$ in~\eqref{eq:SCC-rewritten}.
\end{example}

\subsection{Multistationarity} \label{sec:mss}

A {\em steady state} of a mass-action system $(G,\kappa)$ is some $x^*\in \R_{\geq 0}^n$ at which the right-hand side of the ODEs \eqref{eq:mass_action_ODE} vanishes: $f_{\kappa}(x^*)=0$.  Of particular importance are {\em positive} steady states $x^*\in \R_{> 0}^n$. 
Finally, a steady state  $x^*\in \mathbb{R}_{\geq 0}^n$ 
is {\em nondegenerate} 
if $\Im(df_{\kappa}(x^*)|_S)=S$, where $S$ is the stoichiometric subspace and $df_{\kappa}(x^*)$ is the Jacobian matrix of $f_\kappa$ evaluated at $x^*$.

\begin{definition}[Multistationary] ~
\begin{enumerate}
    \item 
    A mass-action system $(G,\kappa)$ is {\em multistationary} (respectively, {\em nondegenerately multistationary})
 if there exists a stoichiometric compatibility class~\eqref{eq:SCC} that contains
two or more positive steady states  (respectively, nondegenerate positive steady states).  
    \item A reaction network $G$ is {\em multistationary}  (respectively, {\em nondegenerately multistationary}) if there exists a vector of positive rate constants $\kappa$ such that $(G,\kappa)$ is multistationary  (respectively, nondegenerately multistationary). 
\end{enumerate}
\end{definition}

\begin{example}[Example~\ref{ex:cons-law}, continued] \label{ex:ODEs-2}
We return to the network
$ \{ 2A + B \overset{\kappa_1}{\to} 3A,~ A \overset{\kappa_2}{\to} B\}$.
This network is known to be multistationary~\cite{joshi-shiu-small} (see also \cite[Example~1]{cyclic}).
Indeed, when $\kappa_1=\kappa_2$, it is straightforward to check that, 
from the ODEs~\eqref{eq:ODEs-running-example}, the
stoichiometric compatibility class~\eqref{eq:SCC-for-2-species}
contains two positive steady states: $(0.5, 2)$ and $(2, 0.5)$.  
\end{example}

The remainder of this section is devoted to recalling 
results pertaining to multistationary networks with only one or two species~\cite{joshi-shiu-small}.

\begin{lemma}[Multistationarity in 1-species networks with few reactions] \label{lem:1-species-mss}
Assume $r \in \{1,2,3\}$.
Let $G$ be a network with only $1$ species and exactly $r$ reactions, which we write as 
    $\{m_1 A \to p_1 A,~ 
    %m_2 A \to p_2 A , 
    \dots~,~ m_r A \to p_r A \}$, where 
    $m_1  
    \leq \dots \leq m_r$.
\begin{enumerate}
    \item If $r=1$ or $r=2$, then $G$ is not multistationary.
    \item If $r=3$, then $G$ is multistationary if and only if $m_1< m_2< m_3$ and additionally one of the following holds:
    \begin{enumerate}[(a)]
    \item $m_1<p_1$, $m_2>p_2$, and $m_3<p_3$; or
    \item $m_1>p_1$, $m_2<p_2$, and $m_3>p_3$.
    \end{enumerate}
\end{enumerate}
\end{lemma}

The proof of Lemma~\ref{lem:1-species-mss}, in~\cite{joshi-shiu-small},
relies on the classical {\em Descartes' rule of signs}, which we recall states that the number of positive roots of a univariate polynomial $f$, counted with multiplicity, is at most the number of sign changes in the list of coefficients  (with zeroes removed) of $f$. %In the next section, we recall  additional facts about polynomials.

\begin{example} \label{ex:3-rxns}
    By Lemma~\ref{lem:1-species-mss}, the network 
    $
    \{ A \overset{\kappa_1}{\to} 0,~ 2A \overset{\kappa_2}{\to} 3A,~ 4A \overset{\kappa_3}{\to} 3A \}
    =
    \{0 \overset{\kappa_1}{\leftarrow} A,~ 2A \overset{\kappa_2}{\to} 3A \overset{\kappa_3}{\leftarrow} 4A \}$
    is multistationary (we have $r=3$,  $(m_1,m_2,m_3)=(1,2,4)$, and $(p_1,p_2,p_3)=(0,3,3)$).
 \end{example}

To state the next result, which pertains to networks with two species and two reactions, we must first recall the concepts of reactant polytopes (Newton polytopes)~\cite{GMS2} and box diagrams~\cite{joshi-shiu-small}. % [cite MST here].

%----------------------------------
% DEF: reactant polytope
%----------------------------------
\begin{definition} \label{def:NP}
Let $G$
 be a reaction network with $n$ species.
 \begin{enumerate}
     \item  The {\em reactant polytope} of $G$ is the convex hull of (i.e., the smallest convex set containing) the reactants of $G$ (in $\mathbb{R}^n$).
     \item 
     Assume $G$ has exactly two species ($n=2$) and two reactions, $y \to y'$ and $\tilde y \to \tilde y'$, such that the reactant vectors differ in both coordinates 
     (i.e., writing $y=(y_1, y_2)$ and $\tilde y=(\tilde y_1, \tilde y_2)$, then both $y_1 \neq \tilde y_1$ and  $y_2 \neq \tilde y_2$).  
     The {\em box diagram} of $G$ is the rectangle in $\mathbb{R}^2$ for which: 
	\begin{enumerate}
	\item the edges are parallel to the axes of $\mathbb{R}^2$, and
	\item the reactants $y$ and $\widetilde y$ are two opposite corners of the rectangle.
	\end{enumerate}	
 \end{enumerate}
\end{definition}

\begin{example}[Example~\ref{ex:ODEs-2}, continued] \label{ex:box-diagram}
The box diagram of our running example~\eqref{eq:2-species} is shown below, together with its reactant polytope (the marked diagonal of the rectangle):
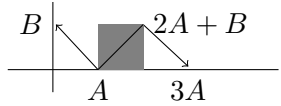
\begin{center}
	\begin{tikzpicture}[scale=.60]
   % axes
	\draw (-1,0) -- (5,0);
	\draw (0,-.5) -- (0,1.5);
  % reactions
	\draw [->] (1,0) -- (0.07,1);
	\draw [->] (2,1) -- (3,0.07);
  % box
	\path [fill=gray] (1,0) rectangle (2,1);
  % diagonal (NP)
	\draw (1,0) --(2,1); 
    % labels
    \node [below] at (1,0) {$A$};
    \node [left] at (0,1) {$B$};
    \node [right] at (2,1) {$2A+B$};
    \node [below] at (3,0) {$3A$};
	\end{tikzpicture}
\end{center}
\end{example}

The following result is due to Joshi and Shiu~\cite[Theorem~4.5]{joshi-shiu-small}.

%----------------------------------
% LEMMA: 2 species and 2 reactions
%----------------------------------
\begin{lemma}[Nondegenerate multistationarity in networks with 2 species and 2 reactions] \label{lem:2-species-2-rxn-mss}
Consider a network $G$ with exactly two species and two reactions, denoted by $y \to y'$ and $\widetilde y \to \widetilde y'$.  Then $G$ is nondegenerately multistationary if and only if the following hold:
    \begin{enumerate}
    % ARROWS MUST BE OPPOSITE TO HAVE MSS
    \item the reaction vectors are negative scalar multiples of each other, that is, $y'-y =- \lambda (\widetilde y' - \widetilde y)$ for some $\lambda \in \mathbb{R}_{>0}$, 
    \item 
	% THE `BOX DIAGRAM' MUST HAVE CERTAIN FORMS
    the reactants $y$ and $\widetilde y$ differ in both coordinates (so the box diagram of $G$ is defined), 
    \item the slope of the reactant polytope is {\em not} $-1$,  and
    \item  the box diagram of $G$ has one of the following ``zigzag'' forms:
% ------ 4 box diagrams
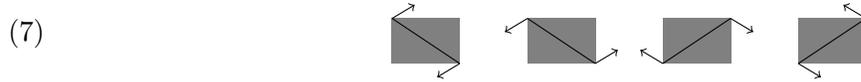
\begin{center}
\begin{equation} \label{eq:boxes}
\begin{tikzpicture}[baseline=(current  bounding  box.center),scale=.6]
%	\begin{tikzpicture}[scale=.6]
  %-------
  % box 1
  %-------
	\path [fill=gray] (0,0) rectangle (1.5,1);
  % reactions
	\draw [->] (0,1) -- (.5,1.3);
	\draw [->] (1.5,0) -- (1,-.3);
  % diagonal (NP)
	\draw (0,1) --(1.5,0); 
  %-------
  % box 2
  %-------
	\path [fill=gray] (3,0) rectangle (4.5,1);
  % reactions
	\draw [->] (3,1) -- (2.5,.7);
	\draw [->] (4.5,0) -- (5,.3);
  % diagonal (NP)
	\draw (3,1) --(4.5,0); 
  %-------
  % box 3
  %-------
	\path [fill=gray] (6,0) rectangle (7.5,1);
  % reactions
	\draw [->] (6,0) -- (5.5,.3);
	\draw [->] (7.5,1) -- (8, .7);
  % diagonal (NP)
	\draw (6,0) --(7.5,1); 
  %-------
  % box 4
  %-------
	\path [fill=gray] (9,0) rectangle (10.5,1);
  % reactions
	\draw [->] (9,0) -- (9.5, -.3);
	\draw [->] (10.5,1) -- (10,1.3);
  % diagonal (NP)
	\draw (9,0) --(10.5,1);
	\end{tikzpicture}
\end{equation}

\end{center}
% ------ 
\end{enumerate}
\end{lemma}

\begin{example}[Example~\ref{ex:box-diagram}, continued] \label{ex:box-diagram-illustrated}
The box diagram of 
the network 
$\{ 2A + B {\to} 3A,~ A {\to} B\}$ 
was shown in Example~\ref{ex:box-diagram}, 
and it matches the third of the four forms depicted 
in~\eqref{eq:boxes}.
Hence, Lemma~\ref{lem:2-species-2-rxn-mss}
implies that $G$
is multistationary.
\end{example}

Lemma~\ref{lem:2-species-2-rxn-mss} pertains to nondegenerate multistationarity, and the next result considers the remaining (degenerate) cases of multistationarity (for networks with two species and two reactions).  The following lemma is immediate from~\cite[Theorem~4.5]{joshi-shiu-small} and its proof.

%----------------------------------
% LEMMA: 2 species and 2 reactions (degenerate)
%----------------------------------
\begin{lemma}[Degenerate multistationarity in networks with 2 species and 2 reactions] \label{lem:2-species-2-rxn-mss-DEGENERATE}
Consider a network $G$ with exactly two species and two reactions, denoted by $y \to y'$ and $\widetilde y \to \widetilde y'$.  Then $G$ is multistationary but \uline{not} nondegenerately multistationary if and only 
if
the two reaction vectors are negative scalar multiples of each other ($y'-y =- \lambda (\widetilde y' - \widetilde y)$ for some $\lambda \in \mathbb{R}_{>0}$) and additionally  one of the following hold:
    \begin{enumerate}
    \item $G$ satisfies the four conditions listed in Lemma~\ref{lem:2-species-2-rxn-mss}, except that the slope of the reactant polytope equals $-1$, 
    \item the reactant complexes are equal ($y= \widetilde y$),
    \item $y'_1 - y_1=\widetilde y_2 -  y_2=0$, or
    \item $\widetilde y_1 -  y_1=y'_2 - y_2=0$.
    \end{enumerate}
\end{lemma}

%-------------------
%    SECTION
%-------------------
\section{Multistationarity regions} \label{sec:multi-region}

A ``multistationarity region'' refers to the part of a reaction network's parameter space where multiple steady states arise.  
However, there are several choices for what constitutes the parameters:
\begin{enumerate}[(a)]
    \item vectors of positive rate constants $\kappa$ (such ``multistationarity region'' were studied in~\cite{CFMW,a6maya,FKdY2}), 
    \item total-constant vectors $c$ (as in~\cite{conradi2019multistationarity}), and
    \item pairs $(\kappa; c)$ (as in~\cite{bihan-dickenstein-giaroli, bradford2020identifying, parameter-geography,amir-superlevel}).
    \end{enumerate}
We focus on options~(a) and~(c).  Both options have been called ``multistationarity regions'' in the literature, but here we distinguish between the two situations (see Definition~\ref{def:multistationarity-region} below).

\begin{remark} \label{rem:full-dim}
For networks that are full dimensional -- such as networks with only one species (and at least one reaction) -- option~(b) above is vacuous, so options~(a) and~(c) coincide. 
For certain full-dimensional networks, their multistationarity regions were studied in~\cite{feliu-helmer, joshi2013complete}.
\end{remark}

\begin{definition} \label{def:multistationarity-region}
Let $G$ be a reaction network with conservation-law matrix $W$.% (of size $d \times n$).  %with $r$ reactions 
\begin{enumerate}
    \item 
    The {\em multistationarity-allowing region} of $G$ 
    is the set of vectors of positive rate constants $\kappa$ for which $(G,\kappa)$ is multistationary.
    \item 
    The {\em multistationarity-enabling region} of $G$ with respect to $W$ is the set of pairs $(\kappa; c)$ of vectors of positive rate constants $\kappa$ and total-constant vectors $c$
    %\in \mathbb{R}^d$ POSITIVE?
    for which $(G,\kappa)$ admits two or more positive steady states in the stoichiometric compatibility class defined by $c$, as in~\eqref{eq:SCC-rewritten}. 
    %(namely, $P= \{x\in {\mathbb R}_{\geq 0}^n \mid Wx=c\}$).
\end{enumerate}
\end{definition}    

Informally, we refer to both types of regions in Definition~\ref{def:multistationarity-region} as ``multistationarity regions''.  
We also use this term when there is no ambiguity, i.e., for full-dimensional networks (recall Remark~\ref{rem:full-dim}).

\begin{remark} \label{rem:nonempty-region}
A network is multistationary if and only if its 
multistationarity-enabling region is nonempty. 
We also note that, in~\cite{telek-feliu-topological},
a pair $(\kappa;c)$ in the multistationarity-enabling region is said to ``enable multistationarity'', which is what inspired our terminology.
\end{remark}

Examples of networks and their multistationarity regions are shown in Table~\ref{tab:networks}. 
These regions are computed in Sections~\ref{sec:examples} and~\ref{sec:results}.  
Before turning our attention to such computations, we first elucidate some properties of multistationarity regions (Section~\ref{sec:basic-ppties}).

\begin{table}[ht]
	\begin{center}
		\begin{tabular}{l c c}
			\hline
			Network & Multistationarity region& Reference \\
			\hline
                %-------------------
                %EXAMPLE
                %-------------------
                $\{ 2A + B \overset{\kappa_1}{\longrightarrow} 3A,~ A \overset{\kappa_2}{\longrightarrow} B\}$ 
                &
                $ c^2 \kappa_1 > 4 \kappa_2$
                & 
                Proposition~\ref{prop:running-example-part-2}
                \\
                %-------------------
                %EXAMPLE
			%-------------------
                $\{
        	A \stackrel[\kappa_1]{\kappa_2}{\leftrightarrows} A + B, 
	           \quad 
        	2B \stackrel{\kappa_3}{\longrightarrow} 3B, 
	           \quad A 
	           \stackrel[\kappa_5]{\kappa_6}{\leftrightarrows} 2A 
	           \}$
                &
	           $\kappa_2^2 \kappa_5 > 4 \kappa_1 \kappa_3 \kappa_6 $
                &
                Proposition~\ref{prop:mss-region-5-rxns}
                \\
                %-------------------
                %EXAMPLE
                % BADAL NETWORK
			%-------------------
                $\left\{ 0 \stackrel[\kappa_1]{\kappa_2}{\leftrightarrows} A,~
    nA \stackrel{\kappa_3}{\longrightarrow} (n + \ell) A
    \right\}$
    &
        $(n-1)^{n-1} \kappa_2^n > n^n \kappa_1^{n-1} \kappa_3 \ell$
    & Proposition~\ref{prop:joshi}
    \\
                %-------------------
                %EXAMPLE
                %-------------------
                $\left\{0 \overset{\kappa_1}{\longleftarrow} A,~ 2A \overset{\kappa_2}{\longrightarrow} 3A \overset{\kappa_3}{\leftarrow} 4A \right\}$
                &
                $4 \kappa_2^3 > 27 \kappa_1^2 \kappa_3$
                &
                Example~\ref{ex:3-rxns-illustrate-theorem} \\
                %-------------------
                %EXAMPLE
			%-------------------
			$      
            \{ 0 \overset{\kappa_{1L}}{\longleftarrow} A 
        {\overset{\kappa_{1R}}{\longrightarrow }} 
        2A 
       \underset{\kappa_{3L}}{\overset{\kappa_{2R}}{\rightleftarrows}}
        3A  \}~$ 
			&
    $\kappa_{1L}> \kappa_{1R},~
    \kappa_{2R}^2 > 4 (\kappa_{1L}- \kappa_{1R}) \kappa_3$
			&
			Example~\ref{ex:mss-region-2-inequalities}
                \\ 
     \hline
		\end{tabular}
	\end{center}
	\caption{Several networks and the inequalities that define the corresponding multistationarity regions. 
    The third network requires $n\geq 2 $ and $ \ell \geq 1$.
 \label{tab:networks}}
\end{table}

\subsection{Basic properties of multistationarity regions} \label{sec:basic-ppties}
It is well known (and follows easily from 
the relevant 
definitions) that the multistationarity-allowing region
is simply a projection of the multistationarity-enabling region, as follows.
\begin{proposition}[Projection of multistationarity regions] \label{prop:projection}
Let $G$ be a reaction network with $r$ reactions and conservation-law matrix $W$. 
Let $\Sigma$ and $\widetilde \Sigma$ be, respectively, the multistationarity-allowing  and multistationarity-enabling regions of $G$ (with respect to $W$).  Then $\Sigma$ is the image of the projection map from $\widetilde \Sigma$ to $\mathbb{R}^r$ given by $(\kappa; c) \mapsto \kappa$.
\end{proposition}

\begin{corollary} \label{cor:projection}
Let $G$ be a reaction network.  
Let $\Sigma$ and $\widetilde \Sigma$ be, respectively, the multistationarity-allowing  and multistationarity-enabling regions of $G$ (with respect to some conservation-law matrix~$W$).  If $\widetilde \Sigma$ is connected, then so is $\Sigma$.
\end{corollary}

\begin{remark}[Converse of Corollary~\ref{cor:projection}] \label{rem:conjecture}
    Feliu and Telek conjectured that the following converse of Corollary~\ref{cor:projection} is true~\cite[\S3]{telek-feliu-topological}: {\em If $ \Sigma$ is connected, then so is $\widetilde \Sigma$.}
    All results in this work are consistent with this conjecture.
\end{remark}

Next, we show that the choice of conservation-law matrix $W$ does not affect the topology of the resulting 
multistationarity region.  This allows us to say that a multistationary region of some network is, for instance, connected, without specifying a choice of $W$.

\begin{proposition}[Choice of conservation-law matrix $W$]
Let $G$ be a reaction network.  Let $W$ and $W'$ both be conservation-law matrices for $G$.   
Let $\Sigma$ and $\Sigma'$ denote the multistationarity-enabling regions of $G$ with respect to $W$ and $W'$, respectively.  
Then $\Sigma'$ is the image of $\Sigma$ under a linear isomorphism of Euclidean space.  Consequently, 
$\Sigma$ and $\Sigma'$ are homeomorphic.
\end{proposition}

\begin{proof}
Let $W$ and $W'$ be conservation-law matrices of a network $G$.  
The rows of both matrices form bases of 
$S^{\perp}$ (where $S$ is the stoichiometric subspace of $G$).  
Hence, there exists an invertible $d \times d$ matrix $M$ such that $W'=MW$.  (Here, $d = n - \dim (S)$, where $n$ is the number of species.)  
Now it is straightforward to check from the relevant 
definitions that the mapping 
$(\kappa; c) \mapsto (\kappa; Mc)$ 
defines a bijection from $\Sigma$ to $\Sigma'$.
%(invertible) linear 
%map 
%$\Sigma \to \Sigma'$
%given by 
%$(\kappa; c) \mapsto (\kappa; Mc)$ yields a linear homeomorphism.
\end{proof}

\subsection{Examples} \label{sec:examples}
In this section, we compute several multistationarity regions.  
%Of course, once we know the multistationarity-enabling region, we can obtain the multistationary-allowing region (via Proposition~\ref{prop:projection}). 
%Nevertheless, here we compute each region separately to illustrate Definition~\ref{def:multistationarity-region}.  We also note that 
Our proofs 
use results on polynomials from Section~\ref{sec:background-polynomials}, and they illustrate key ideas that we use in the next section to prove our main results.
We begin with our running example.  

%\begin{proposition}[Multistationarity-allowing region of the running example] \label{prop:running-example}
%The multistationarity-allowing region of
%$G=\{ 2A + B \overset{\kappa_1}{\to} 3A,~ A \overset{\kappa_2}{\to} B\}$ 
%is all of $\mathbb{R}^2_{> 0}$.    
%\end{proposition}
%
%\begin{proof}
%We must show that $(G,\kappa)$ is multistationary for every choice of $\kappa \in \mathbb{R}^2_{> 0}$. 
%The case of $\kappa_1=\kappa_2$ was already analyzed in Example~\ref{ex:ODEs-2}. 
%Now assume $\kappa_1 \neq \kappa_2$.  
%One can check directly that $(1, \kappa_2/\kappa_1)$ and $(\kappa_2/\kappa_1, 1)$ are (distinct) steady states in the following stoichiometric compatibility class defined by
%$c=1 + \kappa_2/\kappa_1$ (where, as before, the conservation-law matrix is $W=[1 ~1]$):
%%$x(0)=\left( \frac{1}{2}+\frac{\kappa_2}{2 \kappa_1},  ~ \frac{1}{2}+\frac{\kappa_2}{2 \kappa_1} \right)$:
%\begin{align} \label{eq:SCC-for-2-species-another}
%%P_{x(0)} 
%P_c~=~ \left\{ (x_1,x_2) \in \mathbb{R}^2_{\geq 0} \mid x_1 + x_2 = 1+ \frac{\kappa_2}{ \kappa_1} \right\}~.
%\end{align}
%We conclude that, as claimed, $(G,\kappa)$ is multistationary.
%%the multistationarity region of $G$ is $\mathbb{R}^2_{ > 0}$.
%\end{proof}

\begin{proposition}[Multistationarity regions of the running example] \label{prop:running-example-part-2}
The multistationarity-enabling region of the network
$ \{ 2A + B \overset{\kappa_1}{\to} 3A,~ A \overset{\kappa_2}{\to} B\}$ 
with respect to the conservation-law matrix $W=[1 ~1]$, is the following connected set:
\begin{align} \label{eq:multi-region-for-running-ex}
    \left\{ (\kappa_1, \kappa_2, c) \in \mathbb{R}^3_{> 0} \mid
    c^2 \kappa_1 > 4 \kappa_2
    \right\}~,
\end{align}
and hence the 
multistationarity-allowing region equals $\mathbb{R}^2_{> 0}$.    
\end{proposition}

\begin{proof}
From this network's ODEs~\eqref{eq:ODEs-running-example} and the conservation law from $W$, the positive steady states in $P_c$ are the positive intersection points $(x_1,x_2) \in \mathbb{R}^2_{>0}$ of a hyperbola and a line, defined by:
    \begin{align} \label{eq:steady-state-equations-running-ex}
        \kappa_{1}x_{1} x_{2} - \kappa_{2} ~&=~0\\
        x_1 + x_2 ~&=~ c~. \notag
    \end{align}
We substitute $x_2=c-x_1$, from the second equation in~\eqref{eq:steady-state-equations-running-ex}, into the first equation to obtain $h(x_1):= -x_1^2+ cx_1 - \kappa_1/\kappa_2 = 0$.  

Consider the discriminant of $h$, which is $\Delta(h)= c^2- 4 \kappa_2/\kappa_1$.  When $\Delta(h) \leq 0$ (i.e., $c^2 \kappa_1 \leq 4 \kappa_2$), $h$ has at most one real root (counted without multiplicity) and hence the system~\eqref{eq:steady-state-equations-running-ex} does not admit multiple positive roots.
On the other hand, when $\Delta(h) >0$ (i.e., $c^2 \kappa_1 > 4 \kappa_2$), it is straightforward to check that the following are positive steady states in $P_c$:
\begin{align*}
    \left( 
    \frac{c+\sqrt{\Delta(h)}}{2}, ~
    \frac{c-\sqrt{\Delta(h)}}{2}
    \right)
    \quad
    {\rm and}
    \quad 
    \left( 
    \frac{c-\sqrt{\Delta(h)}}{2}, ~
    \frac{c+\sqrt{\Delta(h)}}{2}
    \right)~.
\end{align*}

Next, connectedness of~\eqref{eq:multi-region-for-running-ex} follows from Lemma~\ref{lem:signomial} (applied to 
$g(\kappa_1, \kappa_2, c) :=4 \kappa_2 - c^2 \kappa_1$).
Finally, we apply Proposition~\ref{prop:projection} to~\eqref{eq:multi-region-for-running-ex} to see that the multistationarity-allowing region is $\mathbb{R}^2_{> 0}$.    
\end{proof}

%\begin{remark} \label{rem:altn-proof}
%Ideas behind a more geometric proof of Proposition~\ref{prop:running-example-part-2} appear in \cite[Example~1]{cyclic}.
%\end{remark}

\begin{remark} \label{rem:closely-related-running-ex}
The following network is obtained from the running example by making one reaction reversible:
    $H =\{ 2A + B {\to} 3A,~ A {\rightleftarrows} B\}$.  The multistationarity-enabling region of $H$ was shown to be connected in~\cite{telek-feliu-topological} (and so Corollary~\ref{cor:projection} implies that the multistationarity-allowing region is too).
\end{remark}

Next, we analyze the following (full-dimensional) network:
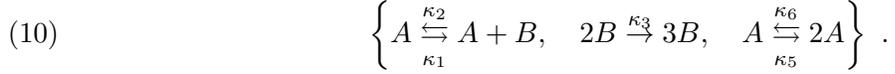
\begin{align} \label{eq:network-5-reactions}
\left\{
	A \stackrel[\kappa_1]{\kappa_2}{\leftrightarrows} A + B, 
	\quad 
	2B \stackrel{\kappa_3}{\rightarrow} 3B, 
	\quad A 
	\stackrel[\kappa_5]{\kappa_6}{\leftrightarrows} 2A 
\right\}~.
\end{align}  
This network~\eqref{eq:network-5-reactions} is obtained by 
removing the reaction labeled by $\kappa_4$ from a network in recent work of Joshi, Kaihnsa, Nguyen, and Shiu~\cite[Example 2.6]{joshi-kaihnsa-nguyen-shiu-1}.  

\begin{proposition}%[Multistationarity region]  
\label{prop:mss-region-5-rxns}
The multistationarity region 
of network~\eqref{eq:network-5-reactions} 
is the following connected set:
\begin{align} \label{eq:multi-region-for-5-rxn-ex}
    \left\{ (\kappa_1, \kappa_2, \kappa_3, \kappa_5, \kappa_6) \in \mathbb{R}^5_{> 0} \mid
         \kappa_2^2 \kappa_5 > 4 \kappa_1 \kappa_3 \kappa_6 
    \right\}~.
\end{align}
\end{proposition}

\begin{proof}
    The network~\eqref{eq:network-5-reactions} generates the following mass-action ODEs~\eqref{eq:mass_action_ODE}:
\begin{align*}
&\frac{dx_1}{dt} ~=~ \kappa_5x_1 - \kappa_6x_1^2\\
&\frac{dx_2}{dt}~ =~ \kappa_1x_1-\kappa_2x_1x_2+\kappa_3x_2^2~.
\end{align*}
The first equation readily yields that $x_1^*=\kappa_5/\kappa_6$ for all positive steady states $(x_1^*, x_2^*)$ (this property is called ``absolute concentration robustness''~\cite{shinar-feinberg}). 
Hence, 
the positive steady states of the mass-action system correspond to positive roots $x_2^*$ of the following 
quadratic in $x_2$:
\begin{align} \label{eq:polyn-5-rxn}
    \kappa_3x_2^2
    -\kappa_2 \left( \kappa_5/\kappa_6 \right) x_2
    +
    \kappa_1 \left( \kappa_5/\kappa_6 \right)~.
\end{align}
Proposition~\ref{prop:num-roots-trinomial} implies that the polynomial~\eqref{eq:polyn-5-rxn} has multiple positive roots if and only if the inequality
$4 \left(\frac{\kappa_1 \kappa_5}{\kappa_3 \kappa_6} \right) - \left( \frac{\kappa_2 \kappa_5}{ \kappa_3 \kappa_6}\right)^2 <0$ holds, which is equivalent to the 
inequality in~\eqref{eq:multi-region-for-5-rxn-ex}.  
Finally, connectedness of the region~\eqref{eq:multi-region-for-5-rxn-ex} comes from applying Lemma~\ref{lem:signomial} to 
$g(\kappa_1, \kappa_2,  \kappa_3, \kappa_5, \kappa_6) :=
4 \kappa_1 \kappa_3 \kappa_6- \kappa_2^2 \kappa_5$.
\end{proof}

We end this section by recalling the multistationarity region for a family of multistationary networks with only one species.  This region was computed by Joshi, as follows~\cite[Lemma~4.3]{joshi2013complete} (and connectedness is immediate from Lemma~\ref{lem:signomial}).

\begin{proposition}[One-species networks with one non-flow reaction] 
\label{prop:joshi}
For the following network:
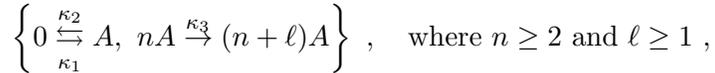
\begin{align*}
    \left\{ 0 \stackrel[\kappa_1]{\kappa_2}{\leftrightarrows} A,~
    nA \stackrel{\kappa_3}{\to} (n + \ell) A
    \right\}~, \quad 
    {\rm where~} n\geq 2 {\rm~and~} \ell \geq 1~, 
\end{align*}
%,  % $\ell_3$
the multistationarity region 
is the following connected set:
\begin{align*} %\label{eq:multi-region-for-joshi-network}
    \left\{ (\kappa_1, \kappa_2, \kappa_3) \in \mathbb{R}^3_{> 0} \mid
    (n-1)^{n-1} \kappa_2^n > n^n \kappa_1^{n-1} \kappa_3 \ell
    \right\}~.
\end{align*}
\end{proposition}

In the next section, we generalize Proposition~\ref{prop:joshi}
to any one-species network with up to three reactions
(see Theorem~\ref{thm:1-species-up-to-3-rxns}).

%-------------------
%    SECTION
%    (RESULTS)
%-------------------
\section{Main results} \label{sec:results}
In this section, we give an example of a
multistationarity region that is disconnected (Proposition~\ref{prop:disconnected}).  
The corresponding network has six reactions and only one species.  In contrast, we show that the multistationarity region is connected for all networks with one species and up to three reactions (Theorem~\ref{thm:1-species-up-to-3-rxns})
and all networks with two species and up to two reactions (Theorem~\ref{thm:2-species}).  

\subsection{Networks with one species} \label{sec:1-species}
This subsection considers networks with only one species.  
Such networks (with at least one reaction ) are full dimensional.

\begin{proposition} \label{prop:disconnected}
The multistationarity region of the following network is disconnected:
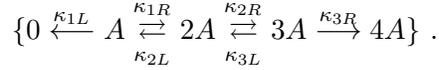
\begin{align*}
        \{ 0 \overset{\kappa_{1L}}{\longleftarrow} A 
        \underset{\kappa_{2L}}{\overset{\kappa_{1R}}{\rightleftarrows}} 
        2A 
       \underset{\kappa_{3L}}{\overset{\kappa_{2R}}{\rightleftarrows}}
        3A \overset{\kappa_{3R}}{\longrightarrow} 4A \}~.
\end{align*}
\end{proposition}

\begin{proof}
The mass-action ODE of this network (where we write $x$ for $x_1$) is:
\begin{align} \label{eq:ODE-6-rxn}
    \frac{dx}{dt} ~=~ 
        \left( \kappa_{1R} - \kappa_{1L} \right) x
        + 
        \left( \kappa_{2R} - \kappa_{2L} \right) x^2
        +
        \left( \kappa_{3R} - \kappa_{3L} \right) x^3~.
\end{align}
By a straightforward application of Descartes' rule of signs, the multistationarity region is a subset of the following disjoint union of two (nonempty) open %and convex % (and hence connected) 
sets $U$ and $V$:
\begin{align*}
    \underbrace{
    \left\{ \kappa \in \mathbb{R}^6_{>0} \mid
        \kappa_{1R} > \kappa_{1L},~
        \kappa_{2R} < \kappa_{2L},~
        \kappa_{3R} > \kappa_{3L}
    \right\}}_{U}
    \dot{\bigcup}
    \underbrace{
    \left\{ \kappa \in \mathbb{R}^6_{>0} \mid
        \kappa_{1R} < \kappa_{1L},~
        \kappa_{2R} > \kappa_{2L},~
        \kappa_{3R} < \kappa_{3L}
    \right\}}_{V} ~,
\end{align*}
where $\kappa:= ( \kappa_{1L}, \kappa_{1R}, 
 \kappa_{2L}, \kappa_{2R},
 \kappa_{3L}, \kappa_{3R})$. 
It now suffices to show that the multistationary region contains a point in $U$ and also contains a point in $V$.  Accordingly, we choose the following:
\begin{align*}
    \kappa(U) ~:=~(1,3,4,1,1,2) \quad \mathrm{and} \quad 
    \kappa(V) ~:=~(3,1,1,4,2,1)~.
\end{align*}
For these vectors of rate constants, the right-hand side of the ODE
~\eqref{eq:ODE-6-rxn} is, respectively, $(x-1)(x-2)$ or $-(x-1)(x-2)$, each of which has $2$ positive roots.  Hence, 
$\kappa(U) \in U$ 
and 
$\kappa(V) \in V$ 
are both in the multistationarity region, which completes the proof.
\end{proof}

\begin{remark}
    As mentioned in the Introduction, only one other network is known to have disconnected multistationarity region: a network modeling ``allosteric reciprocal enzyme regulation''~\cite[Figure~2c]{telek-feliu-topological}.  This regulation model has $10$ species and $10$ reactions, and the proof of Telek and Feliu showing the region is disconnected is somewhat involved.  In contrast, our network is simpler (only $1$ species and $6$ reactions), and the proof of Proposition~\ref{prop:disconnected} is short. 
\end{remark}

\begin{theorem}[One species and up to three reactions] \label{thm:1-species-up-to-3-rxns}
For every reaction network with only one species and up to three reactions, the multistationarity region is connected.
\end{theorem}

\begin{proof}
Let $G$ be a $1$-species network with up to $3$ reactions.  If the multistationarity region of $G$ is empty, then this region is vacuously connected.  We therefore assume that $G$ is multistationary. 
By Lemma~\ref{lem:1-species-mss}, $G$ must have one of the following forms (which correspond to sign sequences $(+,-,+)$ and $(-,+,-)$, respectively):
\begin{align*}
    G~&=~ 
    \left\{
    m_1 A \overset{\kappa_1}{\to} (m_1 + \ell_1) A,~
    (m_1+k) A \overset{\kappa_2}{\to} (m_1 + k - \ell_2) A,~
    (m_1+n) A \overset{\kappa_3}{\to} (m_1 + n +  \ell_3) A
    \right\} 
    \quad \textrm{or} \\
    G~&=~ 
    \left\{
    m_1 A \overset{\kappa_1}{\to} (m_1 - \ell_1) A,~
    (m_1+k) A \overset{\kappa_2}{\to} (m_1 + k + \ell_2) A,~
    (m_1+n) A \overset{\kappa_3}{\to} (m_1 + n -  \ell_3) A
    \right\}~,    
\end{align*}
where $0 < k < n$ and $\ell_1, \ell_2, \ell_3 \geq 1$.

In the first case, the mass-action ODE is as follows (where we let $x:=x_1$ to avoid extra indices):
    \begin{align*}
        \frac{dx}{dt} ~=~
        \kappa_1 \ell_1 x^{m_1} - 
        \kappa_2 \ell_2 x^{m_1+k} +
        \kappa_3 \ell_3 x^{m_1+n}~.  
    \end{align*}
In the second case, the ODE is the same, except that the right-hand side is negated.  It is now straightforward to see that (in both cases) multistationarity occurs precisely when the following trinomial has more than one positive root: 
\begin{align*}
    x^n - \left( \frac{\kappa_2 \ell_2}{\kappa_3 \ell_3} \right) x^k + \left( \frac{\kappa_1 \ell_1}{\kappa_3 \ell_3} \right)~.
\end{align*}
So, by 
Proposition~\ref{prop:num-roots-trinomial} and straightforward algebraic manipulations, we conclude that the multistationarity region of $G$ is the following set: %{\color{red} Include this in the stmt of the result?}
\begin{align} \label{eq:multi-set}
    \left\{
    (\kappa_1, \kappa_2, \kappa_3) \in \mathbb{R}^3_{>0}
    \mid 
    n^{N} (\kappa_1 {\ell_1})^{N-K}
        (\kappa_3 {\ell_3})^K 
        -
        (n-k)^{N-K} k^{K} (\kappa_2 {\ell_2})^{N}
        ~<~ 0
    \right\}~,
\end{align}
 where
 \(d = \gcd(n,k)\), \(N = n/d\), and \(K=k/d\).
Finally, observe that the region~\eqref{eq:multi-set} is the set $g^{-1}(\mathbb{R}_{<0})$, where $g:\mathbb{R}_{>0}^{3} \rightarrow \mathbb{R} $
is given by $
g(\kappa_1, \kappa_2, \kappa_3):=
n^{N} (\kappa_1 {\ell_1})^{N-K}
        (\kappa_3 {\ell_3})^K 
        -
        (n-k)^{N-K} k^{K} (\kappa_2 {\ell_2})^{N}$,
which has exactly one negative term. 
Hence, Lemma~\ref{lem:signomial} implies that the multistionarity region~\eqref{eq:multi-set} is connected.
\end{proof}

\begin{example}[Example~\ref{ex:3-rxns}, continued] \label{ex:3-rxns-illustrate-theorem}
    For the network $\{0 \overset{\kappa_1}{\leftarrow} A,~ 2A \overset{\kappa_2}{\to} 3A \overset{\kappa_3}{\leftarrow} 4A \}$,
 the multistationarity region is as follows, which can be read from~\eqref{eq:multi-set} in the proof of Theorem~\ref{thm:1-species-up-to-3-rxns}:
 \begin{align*} %\label{eq:multi-set-3-rxn}
    \left \{
    (\kappa_1, \kappa_2, \kappa_3) \in \mathbb{R}^3_{>0}
    \mid 
    4 \kappa_2^3 > 27 \kappa_1^2 \kappa_3
    \right\}~.
    \end{align*}
 \end{example}

\begin{example} \label{ex:joshi-again}
For a network of the form $\left\{ 0 {\leftrightarrows} A,~
    nA {\to} (n + \ell) A
    \right\}$, where $n\geq 2$ and $\ell \geq 1$,
the multistationarity region from~\eqref{eq:multi-set} 
exactly matches the one (due to Joshi) shown in Proposition~\ref{prop:joshi}.    
\end{example}

All multistationarity regions that we fully computed thus far are either empty, the full positive orthant, or defined by a single discriminantal inequality.  
In contrast, 
the multistationarity region in the next example is defined by two inequalities.

\begin{example} \label{ex:mss-region-2-inequalities}
Consider the following subnetwork of the network in Proposition~\ref{prop:disconnected}:
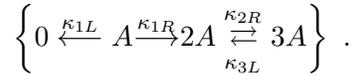
\begin{align*}
        \left\{ 0 \overset{\kappa_{1L}}{\longleftarrow} A 
        {\overset{\kappa_{1R}}{\longrightarrow }} 
        2A 
       \underset{\kappa_{3L}}{\overset{\kappa_{2R}}{\rightleftarrows}}
        3A  \right\}~.
\end{align*}
We claim that the multistationarity region 
%(which is computed readily using ideas appearing already in this section) 
is as follows:
 \begin{align} \label{eq:region-2-ineq} 
    \left\{
    (\kappa_{1L}, \kappa_{1R}, \kappa_{2R}, \kappa_{3L}) \in \mathbb{R}^4_{>0}
    \mid 
    \kappa_{1L}> \kappa_{1R},~
    \kappa_{2R}^2 > 4 (\kappa_{1L}- \kappa_{1R}) 
    \kappa_{3L}
        \right\}~.
    \end{align}
Indeed, this region is readily computed via Proposition~\ref{prop:num-roots-trinomial} and Descartes' rule of signs.  Next, we claim that the region~\eqref{eq:region-2-ineq} is connected.  An outline for the proof is as follows.  First, for ease of notation, we rewrite the region as $    \left\{
    (b, c, d, e) \in \mathbb{R}^4_{>0}
    \mid 0 < (b-c) <     \frac{d^2}{4e}
        \right\}$.  
Given a point $P=(b,c,d,e)$ in the region, the line segment from $P$ to $(\widetilde b, 1,d,e):=(b-c+1,1,d,e)$ remains in the region.  Next, the path from $(\widetilde b, 1,d,e)$ to  $(\widetilde b, 1,\widetilde d,1):= (\widetilde b, 1,d/\sqrt{e},1)$, given by the map $[0,1] \to \mathbb{R}^4$ defined by $t \mapsto \left( \widetilde b, 1,d\sqrt{\frac{e(1-t)+t}{e}}, e(1-t)+t \right)$,  also remains in the region.  We have reduced our problem to checking that the following set is connected: $
 \left\{
    (\widetilde b, \widetilde d) \in \mathbb{R}^2_{>0}
    \mid 0 < (\widetilde b-1) <     \frac{\widetilde{d}^2}{4}
        \right\}$.  This final set is readily seen to be connected, which completes the proof.
\end{example}

\begin{remark}
    In
    Example~\ref{ex:mss-region-2-inequalities}, the multistationarity region~\eqref{eq:region-2-ineq}
    is cut out by two inequalities (as a subset of the positive orthant), and the argument that we gave to show that the region is connected is somewhat ad-hoc.  In the future, we desire results, analogous to Lemma~\ref{lem:signomial}, that handle such situations (for instance, regions cut out by two polynomials of some special form).
\end{remark}

%------------------------
\subsection{Networks with two species} \label{sec:2-species}
%------------------------
The following is the main result of this subsection.
\begin{theorem}[Two species and up to two reactions] \label{thm:2-species}
If $G$ is a network with exactly two species and one or two reactions, then the 
multistationarity-allowing and 
multistationarity-enabling regions of $G$ are connected. 
\end{theorem}
Theorem~\ref{thm:2-species} follows directly from Propositions~\ref{prop:2-species-2-rxns} and~\ref{prop:2-species-2-rxns-DEGENERATE} (below) and the fact that networks with only one reaction are not multistationary (so their multistationarity regions are empty and hence connected).  
We begin with 
Proposition~\ref{prop:2-species-2-rxns}, which can be viewed as generalizing what we proved earlier about the running example, $\{ 2A + B \to 3A,~ A \to B\}$ (Proposition~\ref{prop:running-example-part-2}).  
%This property generalizes to all nondegenerately multistationary networks with exactly two species and two reactions, as follows.

\begin{proposition}[Two species and two reactions] \label{prop:2-species-2-rxns}
Let $G$ be a network with exactly two species and two reactions.  If $G$ is nondegenerately multistationary, then the following hold:
\begin{enumerate}
    \item the multistationarity-allowing region is 
    all of the positive orthant (namely, $\mathbb{R}^2_{>0}$), and 
    \item the multistationarity-enabling region is connected.
\end{enumerate}
\end{proposition}

\begin{proof}
Assume that $G$ is nondegenerately multistationary and has exactly two species and two reactions, which we denote by $y \overset{\kappa}{\to} y'$ and $\widetilde y \overset{\widetilde \kappa}{\to} \widetilde y'$ (to avoid excessive indices).  
By Lemma~\ref{lem:2-species-2-rxn-mss}(1), the two reaction vectors are related by 
$y'-y =- \lambda (\widetilde y' - \widetilde y)$ for some $\lambda \in \mathbb{R}_{>0}$.

Next, we use the following notation from the proof of~\cite[Theorem 4.8]{joshi-shiu-small}.  Let $\gamma$ denote the slope of the reaction vector $y \to y'$, and let $\alpha$ be the slope of the reactant polytope of the network $G$:
\begin{align} \label{eq:gamma-and-alpha}
	\gamma~:=~\frac{y_2'-y_2}{y_1'-y_1}
	\quad \quad {\rm and}
    \quad \quad 
	\alpha~:=~ \frac{\widetilde y_2-y_2}{\widetilde y_1 - y_1}~.
	\end{align}
By Lemma~\ref{lem:2-species-2-rxn-mss}(2--3), 
the denominator of $\alpha$ does not vanish, and, additionally, $\alpha \notin \{0, -1\}$. 
Also, the ``zigzag'' pattern guaranteed by Lemma~\ref{lem:2-species-2-rxn-mss}(4) ensures that $\gamma$ and $\alpha$ have opposite (nonzero) signs.

Next, following the proof of~\cite[Theorem 4.8]{joshi-shiu-small} (or by straightforward algebraic manipulation), the equation defining each stoichiometric compatibility class is as follows:
	\begin{align} 
    \label{eq:stoic-rewritten}
	x_2 &~=~ g(x_1) ~:=~ \gamma x_1 + {c}/({y_1'-y_1})~ \quad \quad \mathrm{for~some~}c \in \mathbb{R}~,
 	\end{align}
and the steady-state equation is:
	\begin{align} 
    x_2 &~=~ h(x_1) ~:=~ K x_1^	 {-1/\alpha}   ~, \label{eq:steady-state-rewritten} %
	\end{align}
where $K:=\left( {\lambda \widetilde \kappa }/ {\kappa }\right)^{1/(y_2-\widetilde y_2)}$. 
In other words, the positive steady states of $G$ correspond to positive roots $x_1\in \mathbb{R}$ of the equation $h(x_1)=g(x_1)$.  

We consider three cases, based on the value of $\alpha$, which we depict qualitatively in Figure~\ref{fig:3-cases}.

\begin{center}
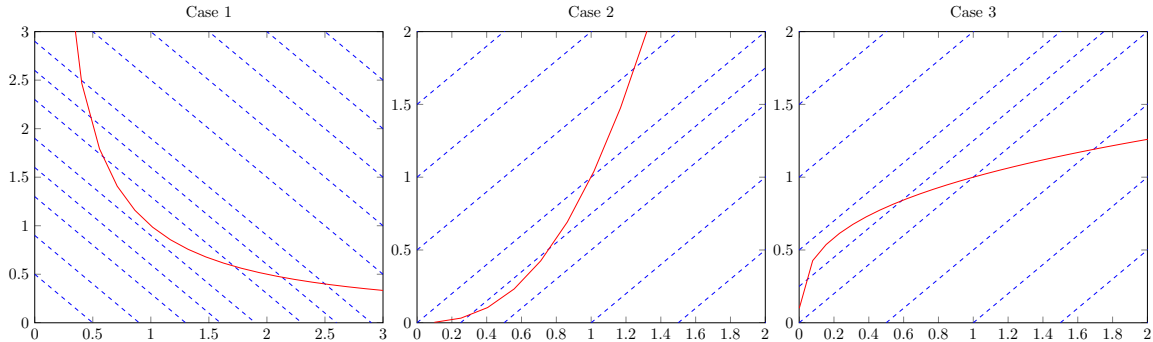
\begin{figure}[ht]
%-----------
% CASE 1
%-----------
\begin{minipage}{.3\textwidth}
\begin{tikzpicture}[scale = 0.55]
\begin{axis}[xmin=0,xmax=3,ymin=0,ymax=3,
    title={Case 1}]
\addplot +[mark=none,color=blue,dashed] coordinates {(0,0.5) (0.5,0)};
\addplot +[mark=none,color=blue,dashed] coordinates {(0,0.9) (0.9,0)};
\addplot +[mark=none,color=blue,dashed] coordinates {(0,1.3) (1.3,0)};
\addplot +[mark=none,color=blue,dashed] coordinates {(0,1.6) (1.6,0)};
\addplot +[mark=none,color=blue,dashed] coordinates {(0,1.9) (1.9,0)};
\addplot +[mark=none,color=blue,dashed] coordinates {(0,2.3) (2.3,0)};
\addplot +[mark=none,color=blue,dashed] coordinates {(0,2.6) (2.6,0)};
\addplot +[mark=none,color=blue,dashed] coordinates {(0,2.9) (2.9,0)};
\addplot +[mark=none,color=blue,dashed] coordinates {(3,0.5) (0.5,3)};
\addplot +[mark=none,color=blue,dashed] coordinates {(3,1) (1,3)};
\addplot +[mark=none,color=blue,dashed] coordinates {(3,1.5) (1.5,3)};
\addplot +[mark=none,color=blue,dashed] coordinates {(3,2) (2,3)};
\addplot +[mark=none,color=blue,dashed] coordinates {(3,2.5) (2.5,3)};
\addplot[
    domain=0.1:3, 
    samples=20, 
    color=red,
    ]{x^(-1)};
\end{axis}
\end{tikzpicture}
\end{minipage}
%-----------
% CASE 2
%-----------
\begin{minipage}{.3\textwidth}
\begin{tikzpicture}[scale = 0.55]
\begin{axis}[xmin=0,xmax=2,ymin=0,ymax=2,
    title={Case 2}]
\addplot +[mark=none,color=blue,dashed] coordinates {(0,0) (3,3)};
\addplot +[mark=none,color=blue,dashed] coordinates {(0,0.5) (2.5,3)};
\addplot +[mark=none,color=blue,dashed] coordinates {(0,1) (2,3)};
\addplot +[mark=none,color=blue,dashed] coordinates {(0.5,0) (3,2.5)};
\addplot +[mark=none,color=blue,dashed] coordinates {(1,0) (3,2)};
\addplot +[mark=none,color=blue,dashed] coordinates {(1.5,0) (3,1.5)};
\addplot +[mark=none,color=blue,dashed] coordinates {(2,0) (3,1)};
\addplot +[mark=none,color=blue,dashed] coordinates {(2.5,0) (3,0.5)};
\addplot +[mark=none,color=blue,dashed] coordinates {(0,1.5) (1.5,3)};
\addplot +[mark=none,color=blue,dashed] coordinates {(0,2) (1,3)};
\addplot +[mark=none,color=blue,dashed] coordinates {(0,2.5) (0.5,3)};
\addplot +[mark=none,color=blue,dashed] coordinates {(0.25,0) (3,2.75)};
\addplot[
    domain=0.1:3, 
    samples=20, 
    color=red,
    ]{x^(2.5)};
\end{axis}
\end{tikzpicture}
\end{minipage}
%-----------
% CASE 3
%-----------
\begin{minipage}{.3\textwidth}
\begin{tikzpicture}[scale = 0.55]
\begin{axis}[xmin=0,xmax=2,ymin=0,ymax=2,
    title={Case 3}]
\addplot +[mark=none,color=blue,dashed] coordinates {(0,0) (3,3)};
\addplot +[mark=none,color=blue,dashed] coordinates {(0,0.25) (2.75,3)};
\addplot +[mark=none,color=blue,dashed] coordinates {(0,0.5) (2.5,3)};
\addplot +[mark=none,color=blue,dashed] coordinates {(0,1) (2,3)};
\addplot +[mark=none,color=blue,dashed] coordinates {(0.5,0) (3,2.5)};
\addplot +[mark=none,color=blue,dashed] coordinates {(1,0) (3,2)};
\addplot +[mark=none,color=blue,dashed] coordinates {(1.5,0) (3,1.5)};
\addplot +[mark=none,color=blue,dashed] coordinates {(2,0) (3,1)};
\addplot +[mark=none,color=blue,dashed] coordinates {(2.5,0) (3,0.5)};
\addplot +[mark=none,color=blue,dashed] coordinates {(0,1.5) (1.5,3)};
\addplot +[mark=none,color=blue,dashed] coordinates {(0,2) (1,3)};
\addplot +[mark=none,color=blue,dashed] coordinates {(0,2.5) (0.5,3)};
\addplot[
    domain=0.001:3, 
    samples=40, 
    color=red,
    ]{x^(1/3)};
\end{axis}
\end{tikzpicture}
\end{minipage}\caption{Three cases of compatibility classes~\eqref{eq:stoic-rewritten} depicted by (blue) dashed lines and the set of steady states~\eqref{eq:steady-state-rewritten} depicted by (red) solid curves.
The three cases correspond to 
$0<\alpha$ (Case 1),
$-1<\alpha<0$ (Case 2), and 
$\alpha<-1$ (Case 3).
In all three graphs, the x-axis and y-axis correspond to the first and second species, respectively.
\label{fig:3-cases}}
\end{figure}
\end{center}

We see in Figure~\ref{fig:3-cases} that, in all three cases, there is a unique compatibility class that is tangent to the steady-state curve. This compatibility class arises from a unique value of $c$ (when $\kappa$ and $\widetilde \kappa$ are fixed), which we call $c^*$, which is as follows (obtained from a straightforward calculus exercise):
%at which there is a unique (degenerate) positive steady state.  This situation arises in the  A straightforward calculus exercise yields the value of $c^*$: 
    \begin{align} \label{eq:c-star}
    c^* ~&=~
    (y_1'-y_1)
    (- \gamma \alpha / K)^{\frac{1}{1+\alpha}} 
    \left(
    K - \gamma  (- \gamma \alpha / K)^{-\alpha}
    \right) \\
    ~&=~
    (y_1'-y_1)
    (- \gamma \alpha 
        (\kappa / \lambda \widetilde \kappa)^{1/(y_2-\widetilde{y}_2)} )^{\frac{1}{1+\alpha}} 
    \left(
    (\lambda \widetilde{\kappa} / \kappa )^{1/(y_2-\widetilde{y}_2)}
    - \gamma  
    (- \gamma \alpha (\kappa / \lambda \widetilde \kappa)^{1/(y_2-\widetilde{y}_2)} )^{-\alpha}
    \right) ~=:~ h(\kappa, \widetilde \kappa)~.
    \notag
    \end{align}

This $c^*$, which we view as a function $h(\kappa, \widetilde \kappa)$, serves as a ``cutoff'' for the range of compatibility classes (equivalently, values of $c$) for which there is more than one (in fact, two) positive steady states.  Indeed, by examining Figure~\ref{fig:3-cases}, we obtain the multistationarity-enabling region
 (with respect to the $(1\times 2)$ conservation-law matrix $W=[-(y_2'-y_2) \quad (y_1'-y_1)]$), which we denote by $\widetilde \Sigma$:
    \begin{align}\label{eq:region-for-2-species}       
    \widetilde \Sigma ~=~ 
    \begin{cases}
        \{(\kappa, \widetilde \kappa, c) \in 
       \mathbb{R}_{>0}^3 \mid 
        c > h(\kappa, \widetilde \kappa)
       \}
        & {\rm if}~ 0<\alpha \rm{~(Case~1)} \\      
        \{(\kappa, \widetilde \kappa, c) \in 
       \mathbb{R}_{>0}^2 \times \mathbb{R} 
        \mid 
        0 < c < h(\kappa, \widetilde \kappa)
       \}
       & {\rm if}~-1<\alpha<0 \rm{~(Case~2)} \\
        \{(\kappa, \widetilde \kappa, c) \in 
       \mathbb{R}_{>0}^2 \times \mathbb{R} 
        \mid 
         h(\kappa, \widetilde \kappa) < c <0
       \}
        & {\rm if}~\alpha<-1 \rm{~(Case~3)}
    \end{cases}
   \end{align}
In all three cases, it is straightforward to see that $\widetilde \Sigma$ is connected. 
Indeed, $\widetilde \Sigma$ is the region above the graph of the positive function $h(\kappa, \widetilde \kappa)$ (in Case~1) or the region between two graphs, one of which lies above the other (Cases~2 and~3).  Finally, the fact that the multistationarity-allowing region equals $\mathbb{R}_{>0}^2$ is verified easily using~\eqref{eq:region-for-2-species} and Proposition~\ref{prop:projection}.
\end{proof}

\begin{example}[Example~\ref{ex:box-diagram-illustrated}, continued]
\label{ex:running-example-end}
 The network
$G=\{ 2A + B \overset{\kappa_1}{\to} 3A,~ A \overset{\kappa_2}{\to} B\}$
falls into Case~1 in the proof of Proposition~\ref{prop:2-species-2-rxns}
($y=(2,1)$, $y'=(3,0)$, $\widetilde y =(1,0)$, $\widetilde{ y}' = (0,1)$, so $\alpha=1$, $\gamma=-1$, $\lambda=1$).   
It is straightforward to compute the ``cutoff function''~\eqref{eq:c-star}:
\[
h(\kappa_1, \kappa_2) ~=~ 2(\kappa_2/\kappa_1)^{1/2}~.
\]
 Hence, from~\eqref{eq:region-for-2-species}, the multistationarity region of $G$ is defined by $c>2(\kappa_2/\kappa_1)^{1/2}$, which is equivalent to the inequality~\eqref{eq:multi-region-for-running-ex} that was computed earlier in Proposition~\ref{prop:running-example-part-2}.
\end{example}

We end by considering the case of degenerate multistationarity.

\begin{proposition}[Two species and two reactions -- degenerate case] \label{prop:2-species-2-rxns-DEGENERATE}
Let $G$ be a network with exactly two species and two reactions.  If $G$ is multistationary but \uline{not} nondegenerately multistationary, then the following hold:
\begin{enumerate}
    \item the multistationarity-allowing region is connected, and 
    \item the multistationarity-enabling region is measure-zero and connected.
\end{enumerate}
\end{proposition}

\begin{proof}
Assume that $G$ is as in the statement of the proposition, and let 
$\Sigma$ 
and 
$\widetilde \Sigma$ 
denote the 
multistationary-allowing 
and 
multistationarity-enabling 
regions, respectively.  
Denote the two reactions of $G$ by $y \overset{\kappa}{\to} y'$ and $\widetilde y \overset{\widetilde \kappa}{\to} \widetilde y'$. By Lemma~\ref{lem:2-species-2-rxn-mss-DEGENERATE}, we know that $y'-y =- \lambda (\widetilde y' - \widetilde y)$ for some $\lambda \in \mathbb{R}_{>0}$, and 
there are four cases to consider:
(Case~1) $G$ satisfies the four conditions of Lemma~\ref{lem:2-species-2-rxn-mss}, except the slope of the reactant polytope equals $-1$; 
(Case~2) %the reactant complexes are equal (
$y= \widetilde y$; %);
 (Case~3) $y'_1 - y_1=\widetilde y_2 -  y_2=0$; 
  and (Case~4) $\widetilde y_1 -  y_1=y'_2 - y_2=0$.

We begin with Case~1. We follow the notation~\eqref{eq:gamma-and-alpha} from a prior proof, where in our case $\alpha=-1$.  
Next, the equations~\eqref{eq:stoic-rewritten}--\eqref{eq:steady-state-rewritten} yield the following: 
\begin{itemize}
    \item the stoichiometric compatibility classes are defined by lines with positive slope, $g(x_1) = \gamma x_1 + {c}/({y_1'-y_1})$, for $c \in \mathbb R$,  and
    \item the steady-state equation is a line through the origin with slope $K$, that is, $h(x_1)=~ K x_1$. %, as $\alpha=-1$).   
\end{itemize}
It follows that multistationarity occurs precisely when $c=0$ and the slopes coincide: $\gamma=K$ (equivalently, 
$\lambda \kappa_2 = \kappa_1 \left( \frac{y_2'-y_2}{y_1'-y_1}\right)^{y_2 - \widetilde y_2} $).  
This yields multistationarity regions that are measure-zero and connected  (where $\widetilde \Sigma$ is with respect to the $(1\times 2)$ conservation-law matrix $W=[-(y_2'-y_2) \quad (y_1'-y_1)]$):
\begin{align*}    
    \widetilde \Sigma ~&=~ 
        \left\{ (\kappa, \widetilde \kappa, c) \in 
       \mathbb{R}_{>0}^2 \times \mathbb{R} \mid 
        c=0,~ \lambda \kappa_2 = \kappa_1 \left( \frac{y_2'-y_2}{y_1'-y_1}\right)^{y_2 - \widetilde y_2}
       \right\}
       \quad {\rm and} \\
        \Sigma ~&=~ 
        \left\{ (\kappa, \widetilde \kappa) \in 
       \mathbb{R}_{>0}^2  \mid 
        \lambda \kappa_2 = \kappa_1 \left( \frac{y_2'-y_2}{y_1'-y_1}\right)^{y_2 - \widetilde y_2}
       \right\}~.
\end{align*}

Next, we consider Cases~2--4 (which are, respectively, subcases (i)--(iii) in the proof of~\cite[Theorem~4.8]{joshi-shiu-small}). We analyze Case~2.  From the proof of~\cite[Theorem~4.8]{joshi-shiu-small} (or direct computation), multistationarity occurs in every compatibility class exactly when 
$\kappa_1=\lambda \widetilde \kappa_2$, 
and there are no positive steady states (in any compatibility class) when $\kappa_1 \neq \lambda \widetilde \kappa_2$.  So, both $\Sigma$ and $\widetilde \Sigma$ are defined by the single equation $\kappa_1=\lambda \widetilde \kappa_2$, and hence are measure-zero and connected.

We turn to Case~3.  We may assume that $\widetilde y_1 - y_1 \neq 0$ (otherwise, we return to Case~2).  
Following the proof of~\cite[Theorem~4.8]{joshi-shiu-small} (or direct computation), the steady-state equation is a vertical line: 
$x_1 = 
\left(\frac{ \kappa }{\lambda \widetilde \kappa}\right)^{1/(y_1 - \widetilde y_1)}$.  
The stoichiometric compatibility classes also are defined by vertical lines $x_1=c$.  Multistationarity occurs exactly when these two lines coincide, which yields the following multistationarity regions (where $\widetilde \Sigma$ is with respect to the conservation-law matrix $W=[1 ~ 0]$):
\begin{align*}    
    \widetilde \Sigma ~&=~ 
        \left\{ (\kappa, \widetilde \kappa, c) \in 
       \mathbb{R}_{>0}^3 \mid 
        c=\left(\frac{ \kappa }{\lambda \widetilde \kappa}\right)^{1/(y_1 - \widetilde y_1)}
       \right\}
       \quad {\rm and} \quad 
        \Sigma ~=~ 
       \mathbb{R}_{>0}^2~.  
       \end{align*}
Notice that $\Sigma$ is connected, and $\widetilde \Sigma$ is measure-zero and connected. 

Finally, Case~4 is symmetric to Case~3.
\end{proof}

\section{Discussion} \label{sec:discussion}
In this work, we proved that one-species networks with six reactions can have multistationarity regions that are disconnected (Proposition~\ref{prop:disconnected}), but this is not the case for one-species networks with up to three reactions (Theorem~\ref{thm:1-species-up-to-3-rxns}).  
The remaining in-between cases are not well understood.  Indeed, it is an open question whether, for one-species networks with four or five reactions (e.g., $\{0 \leftrightarrows A,~ 2A \leftrightarrows 3A\}$), the multistationarity regions
are always connected.  (An answer to this question may involve discriminants of quadrinomials~\cite{otake}.)

We also showed that all two-species networks with up to two reactions have connected multistationarity regions (Theorem~\ref{thm:2-species}).  A future direction is to allow for networks with two reactions, but any number of species (such networks that are multistationary have been classified~\cite{joshi-shiu-small}).  
Other networks for future consideration are 
the eleven ``continuous-flow stirred-tank reactor (CFSTR) atoms of multistationarity'' listed in~\cite[Theorem~5.3]{mss-review}.  
%However, things get messy!  That said, maybe some ``slices'' could be checked (e.g. set some rate constants equal to each other).

Next, one of our contributions was simply to give names to two types of multistationarity regions appearing in the literature: the ``multistationarity-allowing'' and ``multistationarity-enabling'' regions, denoted by $\Sigma$ and $\widetilde \Sigma$.  There is a projection $\widetilde \Sigma \to \Sigma$ that preserves connectivity; 
and Telek and Feliu conjectured the converse, namely, if  $ \Sigma$ is connected, then so is $\widetilde \Sigma$ (Remark~\ref{rem:conjecture}).  

Proving this conjecture would make it easier to check whether $\widetilde \Sigma$ is connected. Indeed, there is a family of biochemical networks (namely, the $m$-site phosphorylation cycles with sequential and distributive mechanisms) for which $\Sigma$ is known to be connected~\cite{FKdY2}, but it is unknown whether $\widetilde \Sigma$ is connected~\cite[\S3]{telek-feliu-topological}.  Our results verify that the conjecture of Telek and Feliu holds for small networks, and future research in this direction 
may provide more evidence for -- and possibly ideas toward proving -- their conjecture.

Finally, our work motivates the question of whether (or when) connectivity of multistationarity regions can be ``inherited'' from small reaction networks to larger ones.  Indeed, such inheritance is known to be possible for certain dynamical properties, including the capacity for nondegenerate multistationarity and/or periodic orbits (see, for instance,~\cite{banaji-inheritance}).  Proving analogous results for the connectivity of multistationarity regions would add significance to our results on small networks.

%*******************************************************************
%Acknowledgements
%*******************************************************************
\subsection*{Acknowledgements} {\small
AS was supported by the NSF (DMS-1752672).
AS thanks Andrea Barton, Elisenda Feliu, Nidhi Kaihnsa, Xiaoxian Tang, and M\'at\'e Telek for helpful discussions.
We are grateful to several reviewers whose detailed suggestions helped improve our work.
}

%\bibliographystyle{plain}
%	\bibliography{bib}

\end{document}